\documentclass[a4paper,10pt,twoside]{amsart}

\usepackage[latin1]{inputenc}
\usepackage{amsmath, amsfonts, amssymb,amsthm}
\usepackage[mathscr]{eucal} 
\usepackage[pdftex]{graphicx}
\usepackage{graphicx}
\graphicspath{ {./images/} }
\usepackage{color}
\usepackage{multicol}
\usepackage[T1]{fontenc}
\usepackage[sc]{mathpazo}
\usepackage[all]{xy}
\usepackage{hyperref}
\usepackage[flenq]{nccmath}

\usepackage{graphicx}
\usepackage{wrapfig}

\usepackage{graphicx}
\graphicspath{ {./Rotationalf0/} }

\usepackage {anysize}
\marginsize{2cm}{2cm}{2cm}{3cm}

\usepackage{enumerate}

\usepackage[pagewise]{lineno}

\linenumbers

\definecolor{alert}{rgb}{0.8,0,0}

\newcommand{\R}{\mathbb{R}}

\newcommand{\s}{\mathbb{S}}
\newcommand{\h}{\mathbb{H}}

\newcommand{\irf}{\R^2\times_h\mathbb{R}}

\newtheorem{theorem}{Theorem}[section]
\newtheorem{proposition}[theorem]{Proposition}
\newtheorem{corollary}[theorem]{Corollary}
\newtheorem{lemma}[theorem]{Lemma}

\theoremstyle{definition}
  \newtheorem{definition}[theorem]{Definition}

\theoremstyle{remark}
\newtheorem{remark}[theorem]{Remark}
\newtheorem{example}[theorem]{Example}

\numberwithin{equation}{section}

\title[ Elliptic  Weingarten  surfaces of minimal type in $\irf$]{ Elliptic Weingarten  surfaces of minimal type in $\irf$}

\author{Carlos Pe\~{n}afiel, Bernardo A. Quaglia and Haimer A. Trejos}

\address{ Instituto de Matem\'{a}tica, Universidade Federal de Rio de Janeiro, Rio de Janeiro, Brazil, 21941-909}
\email{penafiel@im.ufrj.br}

\address{ Instituto de Matem\'{a}tica, Universidade Federal de Rio de Janeiro, Rio de Janeiro, Brazil, 21941-909}
\email{bernardoquaglia@matematica.ufrj.br}

\address{ Instituto de Matem\'{a}tica, Universidade do Estado do Rio de Janeiro, Rio de Janeiro, Brazil, 20550-000}
\email{alexander.serna@ime.uerj.br}

\subjclass[2000]{Primary 53C42; Secondary 53C30}

\keywords{Ellipticity, Weingarten surfaces, Rotationally-invariant, maximum principles, Warped Product.}
\begin{document}
\nolinenumbers

\begin{abstract}
In this paper, we study the elliptic Weingarten of minimal type surfaces immersed in the warped product space $\irf$, when $h$ is a  $C^{1}$-function in $\mathbb{R}^{2}$ with radial symmetry. That is, surfaces whose mean curvature $H$ and extrinsic curvature $K$ satisfy a relationship $H=f(H^{2}-K)$ where $f \in C^{1}(-\epsilon,+\infty)$ with $\epsilon > 0$,  $f(0)=0$ and $4t(f'(t))^{2} < 1$ for $t \in (-\epsilon,\infty)$. We show, under some assumptions about the warping function $h$, the existence and uniqueness of the rotationally-invariant examples of elliptic Weingarten of minimal type surfaces immersed in $\irf$ as well as we study the geometric behavior of its generating curve. 
 \end{abstract}

\maketitle

\section{Introduction}
\label{intro}
The study of the minimal and  constant mean curvature surfaces in Riemannian 3-manifolds is one of topics with most extensively results in the differential geometry. Such surfaces have a variational viewpoint which imply impressive applications about the geometry and topology of these surfaces. The study of minimal surfaces begins with Lagrange in the XIX-century by analyzing the properties of the soap boilers and films. Nowadays, such work is know as  the Euler- Lagrange partial equations, these equations introduce an analytic approach where many examples of minimal surfaces can be defined  and the important techniques of the partial differential equations as the maximum principle and others can be applied in order to get new results of these surfaces. 
Respect to constant mean curvature surfaces, until the beginning  of the XX-century, the sphere was the unique example of such surfaces. Afterwards, A. Alexandrov in \cite{AA} and H. Hopf in \cite{Ho} by using different methods, showed that, under certain conditions about the topology of the surface, the sphere is only one example of a constant mean curvature surface immersed in $\mathbb{R}^{3}$. 

\noindent This framework motivates natural questions about the theory of surfaces immersed into general 3-Riemannian manifolds. In this article, we consider surfaces whose mean curvature and extrinsic curvature are related by a no-trivial equation. Specifically, let $\Sigma$ be a immersed and orientable surface in a 3-Riemannian manifold $M$. Let $H$ and $K$ be the mean curvature and the extrinsic curvature of $\Sigma$, respectively. We say that $\Sigma$ is a Weingarten-type surface if there exists a function $f \in C^{1}(-\epsilon,+\infty)$ with $\epsilon >0$, such that 
\begin{equation*}
    H = f(H^{2}-K).
\end{equation*}
When $f=0$,  $\Sigma$ is a minimal surface and for $f$ a non-zero, constant function, $\Sigma$ is a constant mean curvature surface. The Weingarten surfaces were widely studied by many mathematicians. For example, Hartman and Winter in \cite{HW} and Chern in \cite{CH1} extended the  Hopf's theorem for constant mean curvature surfaces to the case of closed  Weingarten surfaces. Chern in \cite{CH2} and Bryant in \cite{B} proved the classical Liebmann's theorem for Weingarten surfaces. 

\noindent When the function $f$ satisfy the elliptic condition $4tf'(t)^{2} < 1,$ we say that such surface is an elliptic Weingarten surface (EW-surfaces from now on). In \cite{RS} were studied the EW-surfaces immersed into $\mathbb{R}^{3}$, where, by imposing some restrictions over the function $f$, the authors showed that, if $f(0)=0$, the EW-surface is qualitatively a minimal surface and for the case $f(0)>0$, the EW-surface is like a constant mean curvature surface. The elliptic name for such surfaces comes from the fact that in such situation the surface, locally around each point, satisfies a uniformly elliptic almost-linear equation. Thus, the geometrical maximum principle can be applied in order to compare two EW-surfaces. The EW-surface that satisfies  $f(0)=0$ is called  elliptic Weingarten minimal type surface, in short EWMT-surfaces. Moreover, the EW-surface that satisfies $f(0) \neq 0$ is called elliptic Weingarten surface of constant mean curvature type.

In the last thirty years the study of EWMT-surfaces in 3-Riemannian manifolds was a fruitful area in the differential geometry, were many properties of these surfaces were discovered. To mention some results in this topic, in \cite{ST} R. Earp and E. Toubiana showed the existence of EWMT-surfaces of revolution and used it to deduce a half-space-type theorem, which is a generalization of the half-space theorem of Hoffman-Meeks for minimal surfaces, see \cite{HM}. 

\noindent More recently, in \cite{EH}, J. Espinar and H. Mesa extended the theory of minimal surfaces of finite total curvature to EWMT-surfaces of finite total curvature. They showed that a EWMT-surface of finite total curvature immersed in $\mathbb{R}^{3}$ with two ends and embedded outside a compact set must be a rotational surface, this leads to a generalization of the famous Schoen's theorem for minimal surfaces immersed in $\mathbb{R}^{3}$, see \cite{SCH}. Also, they proved that the planes and some special catenoids are the only ones connected EWMT-surfaces immersed in $\mathbb{R}^{3}$ that are embedded outside of a compact set whose total curvature is less than $8 \pi$. 

\noindent Finally, F. Morabito and M. Rodriguez established in \cite{FM} conditions to the existence and uniqueness of the rotational EWMT-surfaces in the product spaces $\mathbb{M}^{2}(\kappa) \times \mathbb{R}$ with the product metric, where, if $\kappa = -1$, $\mathbb{M}^{2}(\kappa)$ label the hyperbolic surface of constant curvature $-1$ and if $\kappa = 1$, $\mathbb{M}^{2}(\kappa)$ label the 2-sphere of constant curvature $1$. They deduced the slab theorem by EWMT-surfaces in these product spaces, generalizing  the classic slab theorem for minimal surfaces immersed $\mathbb{H}^{2} \times \mathbb{R}$ proved in \cite{NEST}.  

In this work, for the product space  $\mathbb{R}^{2} \times \mathbb{R}$, we consider the metric $\overline{g}= g_{2}+e^{2h}g_{1}$ where $g_{2}$ is the usual metric in euclidean space $\mathbb{R}^{2}$, $g_{1}$ is the usual metric in $\mathbb{R}$ and $h$ is a $C^{1}$-function defined in $\mathbb{R}^{2}$ with radial symmetry. In such warped product $\irf$, we study EWMT-surfaces rotationally-invariant about some vertical axis. To do so, first we use a general geometrical maximum principle for EW-surfaces immersed into $\irf$ to show the existence and uniqueness of such EWMT-surfaces immersed in $\irf$. More precisely, under suitable conditions over the warping function $h$, a EWMT-surface, which is rotationally-invariant around some vertical axis, has geometrical properties like a catnoid-type surface. This paper was inspired by previous results due to R. Sa Earp in \cite{BE} and F. morabito and M. Rodriguez in \cite{FM}. The results in this work generalize already know theorems about the existence  of a family of rotational minimal surfaces of catenoidal type immersed in $\mathbb{H}^{2} \times \mathbb{R}$ (see \cite{NR}) the existence of a family rotational EWMT-surfaces immersed in the product spaces (see \cite{FM}).

This paper is organized as follows: in section 2, we introduce the definition and properties of the warped product space $\irf$  that will be use along of this paper. In section 3, we use the maximum principle for EWMT-surfaces immersed in $\irf$ to exhibit a useful barrier to define a EWMT-surface given by a vertical graph.  In section 4, we prove the existence and uniqueness theorem by EWMT-surfaces which are invariant by rotations around one vertical axis. Next, we show that a rotationally-invariant EWMT-surface have a nice geometrical behavior when the warping $h$ has derivative limited by a positive constant. Finally, we exhibit some examples of EWMT-surfaces in $\irf$ in some particular cases of the warping function $h$.  


\section{The warped product space $\irf$.}\label{preliminares}


In this section, we define the warped product $\irf$ and we settle the notations, conventions and the principal facts that will be used along of this paper.

\noindent Let us consider the real line $(\R,g_1)$ with usual metric $g_1=dt^2$ and the Euclidean plane $(\R^2,g_2)$ with metric in the polar form $g_2=d\rho^2+\rho^2d\omega^2$. In the product space $\R^2\times\R$, we can consider the canonical projections 
$$\pi_1:\R^2\times \mathbb{R} \to \R^2 \hspace{.3cm} \textnormal{and} \hspace{.3cm} \pi_2:\R^2 \times \mathbb{R} \to \R$$
and given an arbitrary smooth function $h:\R\to\R$, we can define the Riemannian metric $\overline{g}$ on $\R^2\times\R$ given by the formula (here $e^{\cdot}=\exp(\cdot)$ denotes the exponential, real-function) 
\begin{equation*}
\overline{g}=g_2+e^{2h(\rho)}g_1.
\end{equation*}

\noindent The pair $(\R^2\times\R,\overline{g})$ is called a warped product space and denoted by $\R^2\times_{h}\R$, the function $h$ is called the warping function.  Then, for any $v\in T_{(p,t)}\irf$ at point $(p,t) \in \mathbb{R}^{2} \times \mathbb{R}$, we have
\begin{linenomath*}
$$\overline{g}_{(p,t)}(v,v)=g_2(d\pi_1(v),d\pi_1(v) )+e^{2h(\rho)}g_1(d\pi_2(v),d\pi_2(v)).$$
\end{linenomath*}
$\R^2$ is called the base and $\R$ the fiber of $\irf$ respectively. For our purposes, it is important to compute some geometrical facts about the metric $\overline{g}$. By straightforward  calculations, we obtain the Levi-Civita connection of the metric $\overline{g}$ in $\irf$ in the natural frame $\{\partial_{\rho}, \partial_{\omega}, \partial_{t} \}$, where we have considered the coordinates $(\rho,\omega,t)$ for $\irf$.

\begin{lemma}\label{LC}
For a smooth function $h(\rho)$, we consider the warped product manifold $\irf$ and we denote by $\nabla$ be the Levi-Civita connection in $\irf$ respect of the warped metric $\overline{g}$. Setting $X_1=\partial_\rho$, $X_2=\partial_\omega$, $X_3=\partial_t$. Then, we have: 

\begin{align*} 
\nabla_{X_1}X_1 &=  0  &  \nabla_{X_2}X_1 &=  \frac{1}{\rho} X_{2} & \nabla_{X_3}X_1 &=  h_\rho X_{3} \\ 
\nabla_{X_1}X_2 &= \frac{1}{\rho} X_{2} &  \nabla_{X_2}X_2 &=  -\rho X_{1} & \nabla_{X_3}X_2 &=  0 \\
\nabla_{X_1}X_3 &= h_\rho X_{3} & \nabla_{X_2}X_3 &= 0 & \nabla_{X_3}X_3 &=  -e^{2h}h_{\rho} X_{1}.
\end{align*}

\noindent Where $h_{\rho}$ denotes the derivative of the function $h$ respect to $\rho$.


\end{lemma}

\begin{remark}
From now on, for a function $p(x,z)$ we denote the derivative, for example, with respect to $z$ by $p_z(x,z)$ and for simplicity in many equations, we drop up the point $(x,z)$.     
\end{remark}

 We denote by $\chi(\irf)$ the set of vectors field on $\irf$. It is clear that there exists some privileged vectors fields on $\irf$, in order to highlight some of them, we recall the  definition of a Killing vector field.  Consider $X\in\chi(\irf)$  a vector field,  $q \in \irf$ and $U \subset \irf$ a neighborhood of $q$. Let $\varphi: (-\epsilon,\epsilon) \times U \rightarrow \irf$ be a differentiable application such that for any $ \tilde{q} \in U$, the curve $t \rightarrow \varphi(t,\tilde{q})$ is trajectory of $X$ passing through $\tilde{q}$ at $t=0$. $X$ is called a Killing field, if for each $t_{0} \in (-\epsilon,\epsilon)$, the map $\varphi(t_{0}, ): U \rightarrow \irf$ is an isometry. 

\noindent It is well-know that $X$ is a Killing field if and only if for any pair of vector fields $Y,Z $ on $\irf$, we have  $\overline{g}(\nabla_{Y}X, Z) + \overline{g}(\nabla_{Z}X, Y) = 0$ (see \cite{DC}). Therefore, using the Lemma \ref{LC}, we get the following corollary. 

\begin{corollary}
The warped product space $\irf$ has following geometrical properties: 
\begin{itemize}
    \item The fibers are geodesics of $\irf$.
    \item The vector field $\partial_{t}$ is a killing field and then the vertical translations in $\irf$ are isometries. 
    \item The vector field $\partial_{\omega}$ is a killing field and then the rotations in $\irf$ around to vertical $t$- axis  are isometries. 
\end{itemize}
\end{corollary}

Additionally, another facts about the warped product $\irf$ that will be used are the following:

\begin{itemize}
\item Since the Riemannian manifolds $(\R,g_1)$, $(\R^2,g_2)$ are complete and the expression $e^{h(\rho)}$ never vanish, we have that the warped product $\R^2\times_h\R$ is complete, see  \cite{BO}.    
\item At a point $(\rho_{0}, \omega_{0}, t_{0}) \in \irf$, the fibers $\{ (\rho_0,w_0)\}\times\R$ and the leaves $\R^2\times\{t_0\} $ are orthogonal at $(\rho_0,w_0,t_0)$.
\item For each $t_{0} \in \mathbb{R}$, the slice $S(t_{0}) = \{ (\rho, \omega, t): t= t_{0} \}$ is a totally geodesic surface of
$\irf$.
\end{itemize}


\section{Consequences of the elliptic property}\label{SWS}

 In this section, we enunciate the Hopf maximum principle for  elliptic Weingarten surfaces immersed $\irf$ in the form we shall need. Next, we use it to prove a barrier for the existence of Weingarten surfaces immersed in $\irf$. Finally, we show some algebraic consequences of ellipticity. 

\subsection{Elliptic Weingarten surfaces}\label{SWS}
A Weingarten surface\footnote[1]{Classically labeled W-surface.}  into the warped product $\irf$ is an immersed, oriented surface $\Sigma\looparrowright\irf$ such that its principal curvatures $\kappa_1$ and $\kappa_2$ satisfy a certain relation. More precisely, there exist a non-trivial smooth function $W$ of two variables such that $W(\kappa_1,\kappa_2)=0$. The family of W-surfaces was introduced by Weingarten in the context of finding all isometric surfaces to a given surface of revolution, see \cite{W1} and \cite{W2}. A Weingarten relation can be rewritten as a relation between the mean curvature $H$ and the extrinsic curvature $K$ of the surface, such relation will be denoted in the same way $W(H,K)=0$.  


\noindent We will consider the class of elliptic  Weingarten surfaces, that is, the surfaces whose Weingarten relation can be written as $H=f(H^2-K)$, where $f:(-\epsilon,\infty) \to \R$ with $\epsilon>0$, is a continuous differentiable function  on $(-\epsilon,\infty)$ which satisfies 
\begin{equation}\label{ew1}
4t(f^\prime(t))^2<1, \hspace{.3cm}\forall t\in (-\epsilon,\infty).
\end{equation}
Indeed, this family of surfaces can be divided in two  types, namely,
\begin{itemize}
\item  If $f(0)\neq0$, such surfaces are called of constant mean curvature type. 
\item  If $f(0)=0$, such surfaces are called as elliptic of minimal type.
\end{itemize}  

\noindent In this paper, we will focus on the elliptic Weingarten surfaces of minimal type and we get the following definition.

\begin{definition}
Let $\Sigma\looparrowright\irf$ be an immersed surface, oriented with a globally normal vector field $N$. Then, we say that $\Sigma$ is an elliptic  Weingarten surfaces of minimal type in $\irf$, if the mean curvature function $H=H(N)$ and the extrinsic curvature  function $K$ of the surface $\Sigma$ verifies the equation
\begin{equation}\label{e1}
H=f(H^2-K)
\end{equation} 
where    $f \in C^{1}(-\epsilon,+\infty)$ with $\epsilon >0$, $f(0)=0$ and the function $f$ verifies the condition
\begin{equation*}
\label{e2}
4t(f^\prime(t))^2<1, \hspace{.3cm}\forall t\in (-\epsilon,\infty).
\end{equation*}
From now on, we call such kind of surfaces of EWMT- surfaces, in short. 
\end{definition}

\noindent 

\subsection{Geometrical maximum principles.}

The condition for the surface $\Sigma$ immersed into $\mathbb{R}^{3}$ to be a elliptic Weingarten surface means, in particular, that the surface $\Sigma$ is elliptic. More precisely, if we consider the surface $\Sigma$ as a local graph of a function $u$ around to a fixed point $p\in\Sigma$, it can be showed that the second order partial differential equation on $u$ induced by \eqref{ew1} is an uniformly elliptic almost-linear equation. Thus, a elliptic Weingarten surfaces immersed in $\irf$ satisfies the Hopf  maximum principle. 

\noindent First, we set some notation; let $\Sigma_{1}$ and $\Sigma_{2}$ be immersed and oriented surfaces in $\irf$ and assume  they are given (locally) by vertical graphs of $C^{2}$ functions $u: \Omega \rightarrow \mathbb{R}$ and  $v: \Omega \rightarrow \mathbb{R}$ where $\Omega \subset \mathbb{R}^{2}$ is a bounded domain. Suppose the tangent planes of both, $\Sigma_1$, $\Sigma_2$ agree at a point $p=(\rho_0,\omega_0,t_0)$. Let $H(N_1)$, $H(N_{2})$ be the mean curvature function of $u$ and $v$ with respect to unit normals $N_1$ and $N_2$ that agree at $p$ and let $K_i$ be the extrinsic curvature function of $\Sigma_i$, $i=1,2$. Setting $H_i=H(N_i)$, $i=1,2$ and suppose $\Sigma_i$ satisfy
\begin{equation*}
    H(N_{i})=f(H_{i}^{2}-K_{i})=0, \, \, \, \, i=1,2, 
\end{equation*}
for some elliptic function $f$. With these conditions, we can state the interior maximum principle. 

\begin{theorem}[\bf Interior maximum principle]
Let $\Sigma_{1}$ and $\Sigma_{2}$ be surfaces in $\irf$, which are given by graphs of $C^{2}$ functions $u: \Omega \rightarrow \mathbb{R}$ and  $v: \Omega \rightarrow \mathbb{R}$ where $\Omega \subset \mathbb{R}^{2}$ is a bounded domain and $p=(\rho_0,\omega_0,t_0) \in \Sigma_{1} \cap \Sigma_{2}$ where the tangents planes  of  both surfaces $\Sigma_{1}$ and $\Sigma_{2}$ agree as well as the unit normals $N_1(p)$ and $N_2(p)$.   Suppose that the surfaces $\Sigma_{i}$, $i=1,2$ are elliptic Weingarten surfaces, that is, there exists a real function  $f \in C^{1}(-\epsilon,+\infty)$ with $\epsilon >0$, such that 

\begin{equation*}
    H(N_{i}) = f(H_{i}^{2}-K_i), \, \, \, \, i=1,2,
\end{equation*}
where $H(N_i)$ and $K_i$ are the mean curvature function and extrinsic curvature  of $u$ and $v$ with respect to unit normals $N_i$ of $\Sigma_i$, $i=1,2$. Assume that the function $f$ satisfy the elliptic condition 

\begin{equation*}
    4t(f'(t))^{2}<1 \hspace{.3cm}\forall t\in(-\epsilon,\infty)
\end{equation*}
Then, if $ u(\rho,\omega) \leq v(\rho,\omega)$ when $(\rho,\omega)$ is near to $(\rho_{0},\omega_{0})$, we have that $u=v$ in a neighborhood of $(\rho_{0},\omega_{0})$.  
\end{theorem}

\noindent Now, we focus on the geometrical boundary maximum principle for elliptic Weingarten Surfaces.
\begin{theorem}[\bf Boundary maximum principle]
Let $\Sigma_{1}$ and $\Sigma_{2}$ be surfaces in $\irf$, which are given by graphs of $C^{2}$ functions $u: \Omega \rightarrow \mathbb{R}$ and  $v: \Omega \rightarrow \mathbb{R}$, here $\Omega \subset \mathbb{R}^{2}$ is a bounded domain and $p=(\rho_0,\omega_0,t_0) \in \Sigma_{1} \cap \Sigma_{2}$ where the tangents planes $T_p\Sigma_{1}$ and $T_p\Sigma_{2}$,  the unit normals $N_1(p)$ and $N_2(p)$, and the tangents lines at boundary $T_p\partial\Sigma_{1}$ and $T_p\partial\Sigma_{2}$ agree. The $C^2$-boundaries  $\partial\Sigma_{1}$ and $\partial\Sigma_{2}$ are given by restrictions of $u$ and $v$ to part of the boundary $\partial\Omega$. Suppose that the surfaces $\Sigma_{i}$, $i=1,2$ are elliptic Weingarten surfaces, that is, there exists a real  function  $f \in C^{1}(-\epsilon,+\infty)$ with $\epsilon >0$ such that 

\begin{equation*}
    H(N_{i}) = f(H_{i}^{2}-K_{e}^i), \, \, \, \, i=1,2,
\end{equation*}
where $H(N_i)$ and $K_e^i$ are the mean curvature function and extrinsic curvature  of $u$ and $v$ with respect to unit normals $N_i$ of $\Sigma_i$, $i=1,2$. Assume that the function $f$ satisfy the elliptic condition 

\begin{equation*}
    4t(f'(t))^{2}<1 \hspace{.3cm}\forall t\in(-\epsilon,\infty)
\end{equation*}
Then, if $ u(\rho,\omega) \leq v(\rho,\omega)$ when $(\rho,\omega)$ is near to $(\rho_{0},\omega_{0})$, we have that $u=v$ in a neighborhood of $(\rho_{0},\omega_{0})$.  
\end{theorem}

\begin{remark}
By using the geometrical maximum principle for elliptic Weingarten surfaces, for standard arguments, we conclude that $\Sigma_1=\Sigma_2$.     
\end{remark}
\noindent As an immediate consequence, we show that the only  elliptic Weingarten surfaces immersed in $\irf$ having an extreme value in the height function must be slices.

\begin{corollary}[\bf Barriers]
\label{barriers}
There is no elliptic Weingarten Surfaces $\Sigma_\gamma$, rotationally-invariant with respect to the $t$-axis, whose generating curve $\gamma$ is the graph of a $C^2$ function $t=t(\rho)$ having a local maximum or a local minimum.
    
\end{corollary}
\begin{proof}
The slices $S(t_0)=\R^2\times\{t_0\}$ for some $t_0\in\R$ are totally geodesic surfaces, consequently, EWMT-surfaces. Moreover, $\{S(t_0);t_0\in\R\}$ form a foliation of the warped product $\irf$. If we suppose that such surface $\Sigma_\gamma$ exists then there exists $t_0\in\R$  such that $S(t_0)$ and $\Sigma_\gamma$ satisfies the condition of the geometrical maximum principle. Consequently $\Sigma_\gamma$ have an open part of $S(t_0)$, a contradiction. 
\end{proof}

\noindent This important corollary indicates that we shall need to research by generating curves $\gamma$ which are (locally) graphs of the form $\rho=\rho(t)$.

\subsection{Algebraic consequences of ellipticity}
For the minimal surface case of a surface $\Sigma$ immersed into the Euclidean space $\R^3$, it is well known that its principal curvatures $\kappa_1$ and $\kappa_2$ either have opposite sign at a fixed point $p$ or both vanishes at $p$. Now, we  consider the following (well-known) lemma (see \cite{ST} or \cite{FM}), which establish the same behavior for the principal curvatures of a EWMT-surface immersed into $\irf$.
\begin{lemma}{\cite{FM}}\label{l1}
Let  $f \in C^{1}(-\epsilon,+\infty)$ with $\epsilon >0$ be a function. Then, following statements are equivalent:
\begin{enumerate}
\item $f$ is elliptic and $f(0)=0$.
\item $g(x)=x-f(x^2)$ is a strictly increasing function satisfying $g(0)=0$ and $g^\prime(x)\neq0$ for all $x$.
\item $\overline{g}(x)=x+f(x^2)$ is a strictly increasing function satisfying $g(0)=0$ and $g^\prime(x)\neq0$ for all $x$.
\end{enumerate}
\end{lemma}

\begin{remark}
\label{rm1}
For an elliptic function  $f \in C^{1}((-\epsilon,+\infty))$ with $\epsilon >0$, such that $f(0)=0$, there exists the following limits:
\begin{equation*}
l=\lim_{r\to-\infty}(r-f(r^2)) \hspace{.3cm} \textnormal{and} \hspace{.3cm} L=\lim_{r\to+\infty}(r-f(r^2))
\end{equation*}
where $l\in[-\infty,0)$ and $L\in(0+\infty]$.
\end{remark}  

\noindent As a consequence from this lemma, we have the desired behaviour for the principal curvatures.
\begin{corollary}\label{c1}
Let $\Sigma\looparrowright\irf$ be a EWMT-surface immersed into the warped product $\irf$. Then, its extrinsic curvature function $K$ is a non-positive function, we mean $K(p)\le0$ for all $p\in\Sigma$. Moreover, $K(p)=0$ if, and only if, both principal curvatures $\kappa_1(p)$ and $\kappa_2(p)$ vanish identically at $p$.
\end{corollary}
\begin{proof}

By integrating \eqref{e2} from $0$ to $t$, we obtain $f^2(t)\le t$ for all $t\in (-\epsilon,\infty)$. Then, from equation \eqref{e1}, we get $H(p)\le(\sqrt{H^2-K})(p)$ and finally $K(p)\le0$. Moreover, the Lemma \ref{l1} gives the desired behaviour for the principal curvatures, which conclude the prove.
\end{proof}

\section{ Rotationally-invariant EWMT-surfaces in $\irf$ .}\label{eu}

In the last section, we enunciated a maximum principle for elliptic Weingarten surfaces immersed in $\irf$. As consequence of that, we defined barriers to existence of these surfaces defined by graphs. In this section, we show that  under suitable conditions, it is possible to find rotationally-invariant EWMT-surfaces immersed into $\irf$ and we deduce geometrical properties of these surfaces. We finalize this section with some examples of EWMT-surfaces immersed in $\irf$ by particular cases of the warping function $h$.

 
 \subsection{Existence and uniqueness theorem by EWMT surfaces in $\irf$}
 
 Let  $\rho:(a,b)\to(0+\infty)$, be a strictly positive-function and of $C^2$-class, defined on the open interval $(a,b)$. Let $\gamma^\rho$ be the graph of the function $\rho(t)$ and let $\Sigma_{\gamma^\rho}$ be the surface invariant by rotations around the $t$-axis in the space $\irf$ obtained from the curve $\gamma^\rho$. In this case the surface  $\Sigma_{\gamma^\rho}$  admit the global parametrization $\psi(t,\omega)=(\rho(t),\omega,t)$. Let $H=H(N)$ denotes the mean curvature function of $\Sigma$, here $N$ is the unit, normal,  vector  field  which point in the opposite direction of the revolution axis. Standards computations shows that, in this conditions, $\Sigma_{\gamma^\rho}$ is an EWMT-surface if, and only if, the function $\rho=\rho(t)$ verifies the equation: 

\begin{eqnarray}
	\label{ge1}
	\dfrac{e^h(e^{2h}h_\rho+2h_\rho\rho_t^2-\rho_{tt})}{2(e^{2h}+\rho_t^2)^{\frac{3}{2}}}+ \dfrac{1}{2\rho}\dfrac{e^h}{(e^{2h} + \rho_t^2)^\frac{1}{2}} =f\left(\left[\dfrac{e^h(e^{2h}h_\rho+2h_\rho\rho_t^2-\rho_{tt})}{2(e^{2h}+\rho_t^2)^{\frac{3}{2}}}- \dfrac{1}{2\rho}\dfrac{e^h}{(e^{2h} + \rho_t^2)^\frac{1}{2}} \right]^2 \right),
\end{eqnarray}
where $\rho_t$ and $\rho_{tt}$ denote the first and second derivative of the function $\rho=\rho(t)$ with respect to the variable $t$. By adjusting the arguments presented in \cite{ST} and \cite{FM}, we can show the existence and uniqueness of rotationally-invariant EWMT-surfaces in $\irf$.

\begin{theorem}[Existence and unicity theorem] 
\label{EUT}
Let $\irf$ be a warped product, where $h=h(\rho)$ is a smooth fixed, warping function defined on $\R$. Let $\rho_0>0$ be a positive number, verifying
\begin{equation*}
h_\rho(\rho_0)\cdot\rho_0 \geq 1.
\end{equation*}
Then, there exists only one EWMT-surface $\Sigma_{\rho_0}$ having mean curvature function $H=H(N)$, whose generating curve is the graph of a smooth function $\rho=\rho(t)$ which is defined (locally) on an open domain $I\subset \R$ containing the $0$, the graph of the function $\rho=\rho(t)$ has a local minimal point at $t=0$. More precisely, the function $\rho(t)$ verifies
\begin{equation*}
\rho(t)>0, \hspace{.2cm} \rho(0)=\rho_0, \hspace{.2cm} \rho_t(0)=0.
\end{equation*}  
\end{theorem}

\begin{remark}
Notice that Theorem \ref{EUT}, give us the existence and uniqueness of a one-parameter family (depending of the parameter $\rho_0$) of rotationally-invariant EWMT-surfaces into the warped product $\irf$, whose generating curve is (locally) a graph of a function $\rho=\rho(t)$ and having a (local) minimal point.
\end{remark}

\begin{remark}
The condition over the warping function $h(\rho)$ given in Theorem \ref{EUT}, indicates that the derivative $h_\rho$ is (locally) positive around the point $\rho_0$.
\end{remark}

\subsection{Geometric behaviour of  invariant rotationally-invariant EWMT-surfaces in $\irf$}


 \noindent The Theorem \ref{EUT} shows the existence of a family of EWMT-surfaces immersed into the warped product $\irf$ which are rotationally-invariant and whose generating curve is locally the graph of a positive function $\rho(t)$ having a minimum at $t=0$. Now, we consider the generating curve $\gamma(s)$ ($s$ is the arc-length parameter) of such rotationally-invariant EWMT-surface. The goal of this subsection is give the geometric behaviour of the curve $\gamma(s)$ under some assumptions about the warping function $h$. 
\begin{remark}
 Remind that, the height function of the surface $\Sigma\looparrowright\irf$ is the restriction of the projection
\begin{equation*}
\pi_1:\irf\to \R \hspace{.3cm} \textnormal{given by} \hspace{.3cm} \pi_1(t,p)=t
\end{equation*}
to $\Sigma$. That is, the height function of $\Sigma$ is $\pi_1 \vert_\Sigma$.
\end{remark}

\noindent In order to continue the study of the EWMT-surfaces invariant by rotational movements, now we  consider the smooth curve $\gamma(s)$ parametrized by arclength parameter $s$, lying in the $\rho t$-plane in $\irf$ which is defined by $\Pi=\{(\rho,0,t)\in\irf; \, \omega=0\}$. Thus,
\begin{linenomath*}
$$\gamma(s)=(\rho(s),0,t(s)) \hspace{.5cm} \textnormal{and} \hspace{.5cm} e^{2h(s)}t_s(s)^2+\rho_s(s)^2=1,$$
\end{linenomath*}
where $h(s)=h\vert_{\gamma(s)}$. Therefore, the surface $\Sigma_\gamma$ obtained by rotating the curve $\gamma$ around the $t$-axis can be parameterized by (here $J$ is an open interval)
\begin{linenomath*}
 \begin{equation}
\label{arcapars}
    \psi(s,\omega)=(\rho(s),\omega,t(s)), \hspace{.3cm} s\in J\subset\R,\hspace{.1cm} \omega\in(0,2\pi), \hspace{.1cm} \textnormal{with} \hspace{.1cm} \rho(s)>0.
\end{equation}
\end{linenomath*}
A standard computation give us that the principal curvatures $\kappa_1(s,\omega)=\kappa_1(s)$ and $\kappa_2(s,\omega)=\kappa_2(s)$ of the surface $\Sigma_\gamma$ with respect to the unit normal vector field 
\begin{equation}\label{e6}
N(s,\omega)=\left(-e^ht_s ,0,e^{-h} \rho_s\right)\in T_{\psi(s,\omega)}\irf,
\end{equation}
are given by

\begin{eqnarray}\label{k1k2}
\nonumber\kappa_1(s) & = & e^h\left( [t_{ss}\rho_s-t_s\rho_{ss}] + h_\rho t_s(1+\rho_s^2)\right)(s) \\
 &  & \\
\nonumber\kappa_2(s) & = & \left( \dfrac{e^ht_s}{\rho} \right)(s),
\end{eqnarray}

\noindent where $t_{s}$, $t_{ss}$, $\rho_{s}$ and $\rho_{ss}$ denote the first and  second derivatives of the functions $t(s)$ and $\rho(s)$ respect to $s$. We are supposing that the generating curve $\gamma(s)$ is locally the graph of a function $\rho(t)$, having local minimum at $t=0$, this implies that there exists $s_0\in J$ such that
\begin{equation*}
\dfrac{d\rho}{dt}(0)=0=\dfrac{\rho_s(s_0)}{t_s(s_0)}, \hspace{.5cm} \textnormal{with} \hspace{.5cm} t_s(s_0)\neq 0.
\end{equation*}
from continuity, the function $t_s(s)$ is different from zero in a small interval around $s_0$. The next lemma shows that, actually, this happen for all $s\in J$ when the function warping $h$ is not constant. Recall that, in order to simplify the notations and computations, we omit the variable $s$ from equations  when there is no confusion.


\begin{lemma}\label{l2}
Let $\Sigma_\gamma\looparrowright\irf$ be a rotationally-invariant EWMT-surface immersed into the warped product $\irf$ defined by equation \eqref{arcapars}. Assume that the warping function $h$ is not constant. Then, if $\Sigma_\gamma$ is not a horizontal slice, the height function restrict to $\Sigma_\gamma$ has derivative $t_s=t_s(s)$ which never vanish at an interior point of $\Sigma_\gamma$.
\end{lemma}
\begin{proof}
Suppose that there exist $s^{*}\in \overline{J}$, for some interval $\overline{J}\subset Dom(\gamma)$, here $Dom(\gamma)$ denotes the domain of the generating curve $\gamma(s)$, such that $t_s(s^{*})=0$. Since $s$ is the arclength parameter for $\gamma(s)$, we have 
$$e^{2h(\rho(s))}t_s^2(s)+\rho_s^2(s)=1$$ 
therefore $\rho_s(s^{*})\neq0$. In particular $\rho_s(s)\neq0$ in some interval which we continue to label by $\overline{J}$. The derivative of the expression $e^{2h(\rho(s))}t_s^2(s)+\rho_s^2(s)=1$ gives the condition 
\begin{equation}\label{e7}
e^{2h}t_s(t_{ss}+h_\rho t_s\rho_s)+\rho_s\rho_{ss}=0.
\end{equation}
Using these expression, the principal curvatures $\kappa_1=\kappa_1(s)$ and $\kappa_2=\kappa_2(s)$ becomes (remind $\rho(s)>0$)
\begin{eqnarray}
\label{k33}\kappa_1 & = & e^h\dfrac{t_{ss}}{\rho_s}+2e^hh_\rho t_s=e^h\left( \dfrac{t_{ss}}{\rho_s}+2h_\rho t_s \right),  \\
\label{k44}\kappa_2 & = & \dfrac{e^ht_s}{\rho}. 
\end{eqnarray}

\noindent For any $s\in \overline{J}$, since we consider $\Sigma_\gamma$ oriented by the unit normal vector field $N$ given at \eqref{e6}, the Weingarten equation \eqref{e1} becomes
\begin{equation}\label{e9}
\dfrac{e^{h}}{2}\left(  \dfrac{t_{ss}}{\rho_s}+ \dfrac{t_s}{\rho} +2h_\rho t_s  \right) -f\left(\dfrac{e^{2h}[(t_{ss}+2h_\rho \rho_s t_s)\rho-t_s\rho_s]^2}{4\rho^2\rho_s^2} \right)=0.
\end{equation}
\noindent For $s\in \overline{J}$, equation \ref{e9} can be rewrite as $Q(\rho,\rho_s,t_s,t_{ss})=0$, where 
\begin{equation}
\label{FunctionQ}
Q(u,v,z,w)=\dfrac{e^{h(u)}}{2}\left(  \dfrac{w}{v}+ \dfrac{z}{u} +2h_\rho(u) z  \right) -f\left(\dfrac{e^{2h(u)}[(w+2h_\rho(u) v z)u-zv]^2}{4u^2v^2} \right)=0
\end{equation}
For $u>0$, $v>0$, $z,w\in\R$. A straightforward computation gives
\begin{equation*}
\dfrac{\partial Q}{\partial w}=\dfrac{e^{h(u)}}{2v}\left(1-2\beta f^\prime(\beta^2) \right) \hspace{.3cm} \textnormal{where} \hspace{.2cm} \beta=\dfrac{e^{h(u)}\left[ (w+2h_\rho(u) v z)u-zv \right]}{2uv}.
\end{equation*} 
From ellipticity, we get $1-2\beta f_0^\prime(\beta^2)>0$. Thus, the function $Q=Q(u,v,z,w)$ is strictly increasing (resp. strictly decreasing) with respect to $w$, when restricted to $\{v>0\}$ (resp. $\{v<0\}$). From $\rho_s(s^{*})\neq0$, we have
\begin{equation*}
\dfrac{\partial Q}{\partial w}\left(\rho(s^{*}),\rho_s(s^{*}),t_s(s^{*}),t_{ss}(s^{*}) \right)\neq 0
\end{equation*}
Once here, the implicit function theorem ensures that there exist a $C^1$ function $p$ from a neighborhood of $\left(\rho(s^{*}),\rho_s(s^{*}),t_s(s^{*}) \right)$ in $\R^3$ into a neighborhood of $t_{ss}(s^{*})$ in $\R$ such that $t_{ss}=p\left(\rho,\rho_s,t_s\right)$ in a small interval which we call again $\overline{J}$. If we set $v_1=\rho$, $v_2=\rho_s$, $v_3=t_s$, the Weingarten surface equation $Q(\rho,\rho_s,t_s,t_{ss})=0$ in $\overline{J}$ becomes
\begin{eqnarray}
\nonumber v_1^\prime & = & v_2 \\
\label{e10}  v_3^\prime & = & p(v_1,v_2,v_3) 
\end{eqnarray}
Given initial values $\rho(s^{*})=\rho_{s^{*}}$, $\rho_s(s^{*})=\widehat{\rho}_{s^{*}}$, $t(s^{*})=t_{s^{*}}$, $t_s(s^{*})=0$, the Picard-Lindelof theorem ensures the existence and uniqueness of a solution $(\rho(s),t(s))$ for the ordinary differential system \eqref{e10} with these initial values at $s^{*}$ defined in some $J_2\subset J_1$. By uniqueness,
\begin{equation*}
\rho(s)=\rho_{s^{*}}+\widehat{\rho}_{s^{*}}(s-s^{*}) \hspace{.3cm} \textnormal{and} \hspace{.3cm} t(s)=t_{s^{*}}
\end{equation*}
for any $s\in J_2$. Thus, an open part of $\Sigma_\gamma$ is a piece of the slice $S(s^{*})$ and the maximum principle implies that $\Sigma_{\gamma} \subset S(s^{*})$ so we get a contradiction.
\end{proof}




\begin{remark}
As a immediate consequence from Corollary \ref{barriers}, there is no compact, without boundary EWMT-surface immersed into $\irf$. Moreover, the Lemma \ref{l2} implies that the generating curve $\gamma$ of the complete surface $\Sigma_\gamma$, must be defined in $Dom(\gamma)=\R$. 
\end{remark}

\noindent From now on, we assume that the function $t_s(s)$ is positive for all $s\in Dom(\gamma)=\R$. The following lemma guarantee that the function $\rho(s)$ have a unique extreme value and this is a global minimum, when the warping function $h$ has positive derivative. 
 
 
\begin{lemma}\label{l5}
Let $\Sigma_\gamma\looparrowright\irf$ be a rotationally-invariant EWMT-surface, which is immersed into the warped product $\irf$ having warping function which satisfy $h_\rho(\rho)>0$ for all $\rho\in\R$. Let  $s_0\in\R$ be the number where the function $\rho=\rho(s)$ attains a local minimum. Then, the function $\rho(s)$ has a unique local extreme at $s_{0}$ and $\rho(s_{0})$ is a global minimum.

\end{lemma}

\begin{proof}
From lemma \ref{l2}, $t_s(s)\neq0$. Then, the equations in \eqref{e7} of the principal curvatures $\kappa_1=\kappa_1(s)$ and $\kappa_2=\kappa_2(s)$ becomes
\begin{eqnarray}
\label{k11}\kappa_1(s) & = &\left(\dfrac{e^{2h}h_\rho t_s^2-\rho_{ss}}{e^ht_s}  \right)(s), \\
\label{k22}\kappa_2(s) & = & \left(\dfrac{e^ht_s}{\rho}\right)(s). 
\end{eqnarray}
Now, from corollary \ref{c1}, $\kappa_{1}(s)<0$. Then, the equation \eqref{k11} and the hypotheses  $h_{\rho} > 0$ imply that $\rho_{ss}(s)>0$. So, the function $\rho_{s}(s)$ is a strictly increasing function. Since $\rho_{s}(s_{0}) = 0$, we get $\rho_s(s)<0$ for $s\in(-\infty,s_0)$ and $\rho_s(s)>0$ for $s\in(s_0,+\infty)$. Consequently, $s_{0}$ is the unique global minimum of $\rho(s)$.
\end{proof}


\noindent As consequence of lemma \ref{l5}, we obtain the next corollary. 

\begin{corollary}
\label{rmkri}
Let $\Sigma_\gamma\looparrowright\irf$ be a rotationally-invariant EWMT-surface, which is immersed into the warped product $\irf$ having warping function which satisfy $h_\rho(\rho)>0$ for all $\rho\in\R$. Then, the generating curve $\gamma(s)=(\rho(s),0,t(s))$ has functions $\rho(s)$ and $t(s)$ satisfying the following  properties :
\begin{enumerate}
\item $\rho_s(s)<0$ for $s\in(-\infty,s_0)$ and $\rho_s(s)>0$ for $s\in(s_0,+\infty)$.
\item For $t_s>0$, we have $\kappa_2(s)>0$,  for all $s$, which implies $\kappa_1(s)<0$,  for all $s$.
\item For the function $\rho_{ss}(s)$, we have
\begin{equation}\label{e14}
\rho_{ss}>0, \hspace{.3cm} \forall s.
\end{equation}
\item For the function $t_{ss}(s)$, we obtain
\begin{equation}\label{e15}
t_{ss}=e^{-h}\rho_s( \kappa_1-2h_\rho t_s),
\end{equation}
which implies that $t_{ss}(s)<0$ for all $s>s_0$, $t_{ss}(s_0)=0$ and $t_{ss}(s)>0$ for all $s<s_0$.
\end{enumerate} 
\end{corollary}

\noindent Let us mention some important consequences of lemma \ref{l5} and corollary \ref{rmkri}. First, from the sign of the principal curvatures, we have the next corollary.

\begin{corollary}\label{c2}
Let $\Sigma_\gamma$ be a rotationally-invariant surface EWMT-surface immersed in $\irf$, which is generated by a curve $\gamma$  around the $t$-axis with $t_{s}>0$. Suppose that $h_\rho(\rho)>0$ for all $\rho\in\R$ and $\Sigma_\gamma$ is oriented via the unit normal vector field $N$ given at \eqref{e6}. Then, the generating curve cannot touch the axis of revolution.
\end{corollary}
\begin{proof}
Since $\Sigma_\gamma$ is a regular surface, if $\Sigma_\gamma$  cut the axis of revolution, it must be orthogonally and  the intersection point $p_0$ must be an umbilical point of the surface $\Sigma_\gamma$. Since $\Sigma_\gamma$ is a EWMT-surface, we have that the principal curvatures $\kappa_1(p)=\kappa_2(p)=0$ vanish identically, and this is a contradiction.
\end{proof}

\noindent By using lemma \ref{l5}, a rotationally-invariant surface EWMT-surface have a plane of symmetry. 

\begin{proposition}
\label{Psym}
Let $\Sigma_\gamma\looparrowright\irf$ be a rotationally-invariant EWMT-surface immersed into the warped product $\irf$ defined by equation \eqref{arcapars}. Suppose that the warping function satisfy $h_\rho(\rho)>0$ for all $\rho\in\R$. Let $s_0\in\R$ be the number such that the function $\rho=\rho(s)$ attains its global minimum. Suppose that the curve $\gamma(s)$ is defined in an interval $I= (s_{0}-\delta,s_{0}+\delta)$ possibly $\delta=\infty$. Then, $\Sigma_{\gamma}$ is symmetric respect to slice $S(t(s_{0}))$. More precisely, we have $\rho(s)=\rho(2s_{0}-s)$ and $t(s)=2t(s_{0})-t(2s_{0}-s)$ for any $s \in I$.
\end{proposition}

\begin{proof}
Since $e^{2h(\rho(s))}t_s^2(s)+\rho_s^2(s)=1$ and $\rho_{s}(s_{0}) = 0$, we conclude that $t_{s}(s_{0}) \neq 0$. Define the rotationally-invariant surface $\Sigma_{\tilde{\gamma}}$ generated by the curve $$\tilde{\gamma}(s) = (\tilde{\rho}(s),0,\tilde{t}(s))$$
where $\tilde{\rho}(s)= \rho(2s_{0}-s)$ and $\tilde{t}(s)=2t(s_{0})-t(2s_{0}-s)$ for $s \in I$. On the one hand, $\tilde{t}_{s}(s_{0}) = t_{s}(s_{0}) \neq 0$ in a small neighborhood of $s_{0}$. By other hand, from definition of $h(\rho)$ and $\tilde{\rho}(s)$, we can check that for the function $Q(u,v,z,w)$ in the equation \eqref{FunctionQ}, we have $Q(\tilde{\rho}, \tilde{\rho}_{s}, \tilde{t}_{s}, \tilde{t}_{ss}) = 0$ and then $\Sigma_{\tilde{\gamma}}$ is a EWMT-surface. Moreover, $\tilde{\rho}(s_{0}) = \rho(s_{0})$,  $\tilde{\rho}_{s}(s_{0}) = \rho_{s}(s_{0})$, $\tilde{t}(s_{0}) = t(s_{0})$ and $\tilde{t}_{s}(s_{0}) = t_{s}(s_{0})$. Therefore, we deduce from uniqueness of solution for the equation \eqref{FunctionQ} that the surface $\Sigma_{\gamma}$ and $\Sigma_{\tilde{\gamma}}$ are locally equal  around the point $\gamma(s_{0}) = \tilde{\gamma}(s_{0})$. Finally, from equation \eqref{e6} and the maximum principle, we obtain that $\Sigma_{\gamma} = \Sigma_{\tilde{\gamma}}$ globally.
\end{proof}

\noindent So far we have focused on the study of the derivatives of the functions $\rho(s)$ and $t(s)$, now we focus on the study such functions.


\begin{proposition}\label{l6}
Let $\Sigma_\gamma\looparrowright\irf$ be a rotationally-invariant, EWMT-surface immersed into the warped product $\irf$, where the warping function satisfies $h_{\rho}(\rho)>0$ for all $\rho\in\R$. Then, the function $\rho(s)$ satisfy
\begin{equation*}
\lim_{s\to +\infty}\rho(s)=+\infty \hspace{.3cm} \textnormal{and} \hspace{.3cm} \lim_{s\to -\infty}\rho(s)=+\infty.
\end{equation*}
\end{proposition}
\begin{proof}
From corollary \ref{rmkri}, we have $\rho_{ss}(s)>0$ for all $s\in\R$.  Suppose that the function $\rho(s)$ is bounded on $(s_0,+\infty)$ (on $(-\infty,s_0)$)  where we have $\rho_s(s)>0$ (for the other case, $\rho_s(s)<0$). Since $\rho(s)$ is monotone and bounded function, we get that $\rho(s)\to\rho_1$ as $s\to  \infty$, for some positive constant $\rho_1$. 

\noindent Now, fix $s_{1} \in (s_{0},\infty)$ and define the sequence $s_{n} = s_{1}+n-1$, hence $s_{n} \rightarrow \infty$ when $n \rightarrow \infty$. By the Mean Value Theorem, there exist $\tilde{s}_{n} \in (s_{n}, s_{n+1})$ so that $\rho_{s}(\tilde{s}_{n}) = \rho(s_{n+1}) - \rho(s_{n})$, since $\rho(s_{n}) \rightarrow \rho_{1}$ when $n \rightarrow  \infty$, consequently, the sequence $\rho_{s}(\tilde{s}_{n})$ converges to zero when $n \rightarrow  \infty$. But this a contradiction, because $\rho_{s}(s)$ is a strictly increasing function and this implies that $\rho_{s}(s)$ is far away from 0 when $s$ goes to $\infty$. The case $(-\infty,s_0)$ is analogous and we complete the proof.
\end{proof}


\noindent Finally, we focus on the behavior of the function $t=t(s)$ when the warping function is bounded from below. 

\begin{proposition}\label{l7}
Let $\Sigma_\gamma\looparrowright\irf$ be a rotationally-invariant EWMT-surface immersed into $\irf$. Suppose that the warping  satisfies $h_{\rho}(\rho)>c$ for all $\rho\in\R$, where $c$ is a positive constant. Then, the limits  $\displaystyle{\lim_{s\to +\infty}t(s)}$ and $\displaystyle{\lim_{s\to -\infty}t(s)}$ exists and $\displaystyle{\lim_{s\to -\infty}t(s) = -\lim_{s\to +\infty}t(s)}$.
\end{proposition}

\begin{proof} 

From corollary \ref{rmkri}, $\rho_{s}(s)$ is a strictly increasing function. Moreover, since $$e^{2h(\rho(s))}t_s^2(s)+\rho_s^2(s)=1,$$ then $\rho_{s}(s)$ is bounded by 1 and we conclude that the limit  $\displaystyle{\lim_{s \rightarrow \infty} \rho_{s}(s)}$ exists and this immediately implies that $\displaystyle{\lim_{s \rightarrow \infty} \rho_{ss}(s)} = 0$. Now we want to show the following claim:  

\noindent {\bf Claim.} $\displaystyle{\lim_{s \rightarrow \infty} t_{s}(s) = 0}$.

\noindent In fact; suppose that $\displaystyle{\lim_{s \rightarrow \infty} t_{s}(s) \neq 0}$ and consider the principal curvatures:

\begin{eqnarray*}
\kappa_1(s) & = &\left(\dfrac{e^{2h}h_\rho t_s^2-\rho_{ss}}{e^ht_s}  \right)(s), \\
\kappa_2(s) & = & \left(\dfrac{e^ht_s}{\rho}\right)(s). 
\end{eqnarray*}
Since  $\displaystyle{\lim_{s \rightarrow \infty} \rho_{ss}(s)} = 0$, we have that $k_{1}(s)$ goes to  $h_{\rho} t_{s} e^{h}$ when $s$ goes to infinity. Besides, as $h_{\rho}>c$, $t_{s}(s)$ goes to a positive number when $s \rightarrow \infty$, $h(\rho)$ is a strictly increasing function and  $\displaystyle{\lim_{s \rightarrow \infty} \rho(s)} = \infty$, we obtain that $k_{1}(s)$ is positive for some $s$ bigger enough. In the same way, we see that $k_{2}$ is always a positive function. So, for $s$ bigger enough, the functions $k_{1}(s)$ and $k_{2}(s)$ have positive sign and this  contradicts the corollary \ref{rmkri} and show the claim.

\noindent By other hand, from the above claim, the proposition \ref{l6}, the equation $e^{2h}t_{s}^{2}= 1 - \rho^{2}_{s}$ and the boundeness of $\rho_{s}(s)$ imply that $\kappa_2(s)  =  \left(\dfrac{e^ht_s}{\rho}\right)(s)$ goes to zero when $s$ goes to infinity. In the same way, from corollary \ref{c1},  $\kappa_{1}(s)$ goes to zero whenever $s$ goes to infinity.  Now, since $f \in C^{1}(-\epsilon,\infty)$ and $f(0)=0$, if we choose a small $\delta>0$, then for each $\lvert t \rvert < \delta$, we have $\lvert f(t^{2}) \rvert < \frac{1}{2} \lvert t \rvert$. Hence, for $s$ big enough, we obtain that $\frac{\lvert \kappa_{1}(s) - \kappa_{2}(s) \rvert}{2} < \delta$ and 

\begin{equation*}
- \frac{\lvert \kappa_{1}(s) - \kappa_{2}(s) \rvert }{2} \leq \kappa_{1}(s) + \kappa_{2}(s) = 2f \Big( \frac{ (\kappa_{1}(s) - \kappa_{2}(s))^{2} }{4}  \Big) \leq \frac{\lvert \kappa_{1}(s) - \kappa_{2}(s) \rvert }{2}.   
\end{equation*}
The inequality $\kappa_{1}(s)+ \kappa_{2}(s) \leq  \frac{\lvert \kappa_{1}(s) - \kappa_{2}(s) \rvert }{2}$ and the supposition $t_{s} > 0$ imply $\kappa_{1}(s) \leq -\frac{1}{3}\kappa_{2}(s)$. Therefore, from expressions by $\kappa_{1}(s)$ and $\kappa_{2}(s)$ in equations \eqref{k33} and in \eqref{k44} , we get

\begin{equation*}
\frac{t_{ss}}{\rho_{s}} \leq -\frac{1}{3} \frac{t_{s}}{\rho} - 2 h_{\rho} t_{s},
\end{equation*}



\noindent as $t_{s}>0$ and $\rho_{s}>0$, then

\begin{equation}
\label{IN1}
\frac{t_{ss}}{t_{s}} \leq -\frac{1}{3}  \frac{\rho_{s}}{\rho} - 2h_{\rho} \rho_{s} \leq -2 h_{\rho} \rho_{s} ,    \end{equation}

\noindent  since $\displaystyle{\lim_{s \rightarrow \infty} \rho_{s}(s)}$ is a positive number and $h_{p}> c$, hence from inequality \ref{IN1}, there is a negative number $m$ such that

\begin{equation}
\label{IN11}
\frac{t_{ss}}{t_{s}} < m, \, \, {\text for} \, \,  s > s_{1} 
\end{equation}

\noindent where $s_{1}$ is big enough. By integrating respect to $s$ in the two sides of inequality \eqref{IN11} in the interval $(s_{1},s)$, we conclude that

\begin{equation}
\label{IN2}
t_{s} < ae^{ms},     
\end{equation}

\noindent where $a$ is a positive constant. Now, integrating the inequality \eqref{IN2} in the interval $(s_{1},s)$ and after some calculations, we get

\begin{equation}
\label{IN3}
t(s) < t(s_{1}) + \frac{a}{m} (e^{ms} - e^{ms_{1}} ).    
\end{equation}

\noindent Since $m$ is negative and $t(s)$ is a strictly increasing function, consequently, from inequality \eqref{IN3}, we obtain that $\displaystyle{\lim_{s \rightarrow \infty} {t(s)}}$ exists. Finally, by proposition \ref{Psym}, the limit $\displaystyle{\lim_{s \rightarrow -\infty} {t(s)}}$ exists too and by the symmetry it is equal to $-\displaystyle{\lim_{s \rightarrow \infty} {t(s)}}$. This concludes the proof.
\end{proof}

\noindent We point out that the results in lemma \ref{l5}, corollary \ref{c2} and propositions \ref{Psym}, \ref{l6} and \ref{l7} holds for the warped product $\irf$ when the warping function satisfies $h_{\rho}> c$, where $c$ is positive constant. Consequently, we get the following general result.

\begin{theorem}[Main Theorem]
\label{MT1}
 Let $h: \mathbb{R}^{2} \rightarrow \mathbb{R}$ be a radial symmetry warping function such that its derivative satisfy $h_\rho(\rho)\ge c$ for $\rho\in\R$ and some positive constant $c$, and let $\rho_0>0$ be a number such that $h(\rho_{0}) \cdot \rho_{0} \ge 1$. Then there exists a unique  rotationally-invariant around to $t$-axis EWMT-surface $\Sigma_\gamma^{\rho_0}$ immersed in $\irf$ generated by the curve $\gamma=(\rho(t),0,t)$ which is the graph of the function $\rho(t)$. Moreover, the function $\rho(t)$ has the following properties: 

\begin{enumerate}
    \item the curve $\gamma$ is convex and has a global minimum at $(\rho_0,0,0)$,
    \item the function $\rho(t)$ is defined in some interval $I=(-t_0,t_0)$ for $t_0>0$,
    \item $\gamma$ is asymptotic to the line $t=\pm t_0$ when $t$ goes to $\pm\infty$,
    \item finally, $\gamma$ is symmetric with respect to the slice $\{t=0\}$. 
\end{enumerate}
\end{theorem}

\noindent A rotationally-invariant surface $\Sigma$ immersed $\irf$ that satisfies all properties of theorem \ref{MT1} is called a surface of catenoidal type in $\irf$.

\begin{remark}
A minimal surface can be seen as a elliptic Weingarten surface for which $f$ vanishes identically. 
\end{remark}

\subsection{Explicit examples}
In the following two examples, for the warped product $\irf$ we consider the warping function $h:\R\to\R$ given by $h(\rho)=\rho$ which it is a function that satisfies the conditions of the main theorem.

\begin{example}[Minimal surface]
First, we want to plot the rotationally-invariant minimal surface $\Sigma_\gamma$, whose generating curve $\gamma$ is the graph of the function $\rho=\rho(t)$, where $\rho(t)$ is the function that solve the equation \eqref{ge1}, given by
\begin{equation*}
\left(\rho '(t)^2+e^{2 \rho (t)}\right)+\rho (t) \left(-\rho ''(t)+2 \rho '(t)^2+e^{2 \rho (t)}\right)=0, 
\end{equation*}
with initial conditions $\rho(0)=1$, $\rho_t(0)=0$. By using mathematical software, we obtain that the graph of the generating curve is given in the Figure 1, where we have the interval $I=(-2.335,2.335)$ as the domain of the function $\rho(t)$.

\begin{figure}[h]
\includegraphics[width=7cm, height=4cm]{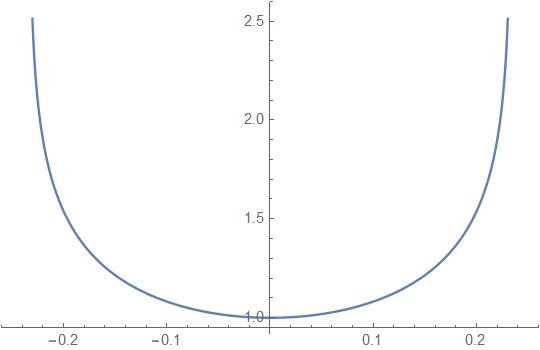}
\centering
\caption{Graph of solution $\rho(t)$ for example 1.}
\end{figure}



\end{example}

\noindent Now we plot a no-minimal EWMT-surface.

\begin{example}[EWMT-surface] In order to obtain a non-minimal EWMT-surface,  we consider the elliptic function $f(t)=\dfrac{1}{2}\sqrt{t}$. In this case, We want to plot the rotationally-invariant EWMT-surface $\Sigma_\gamma$ embedded into $\irf$, whose generating curve $\gamma$ is the graph of the function $\rho=\rho(t)$, where $\rho(t)$ is the function that solve the equation 

\begin{eqnarray*}
& & \frac{1}{2} \left(\frac{e^{\rho (t)}}{\rho (t) \sqrt{\rho '(t)^2+e^{2 \rho (t)}}}+\frac{e^{\rho (t)} \left(-\rho ''(t)+2 \rho '(t)^2+e^{2 \rho (t)}\right)}{\left(\rho '(t)^2+e^{2 \rho (t)}\right)^{3/2}}\right) \\
& - & \frac{1}{2} \sqrt{\left(\frac{e^{\rho (t)} \left(-\rho ''(t)+2 \rho '(t)^2+e^{2 \rho (t)}\right)}{\left(\rho '(t)^2+e^{2 \rho (t)}\right)^{3/2}}-\frac{e^{\rho (t)}}{\rho (t) \sqrt{\rho '(t)^2+e^{2 \rho (t)}}}\right)^2}=0
\end{eqnarray*}
with initial conditions $\rho(0)=1$, $\rho_t(0)=0$. Again, by using a standard mathematical software, we obtain that the graph of the generating curve is given in the Figure 2, where we have the interval $I=(-0.368,0.368)$ as the domain of the function $\rho(t)$.

\begin{figure}[h]
\includegraphics[width=7cm, height=4cm]{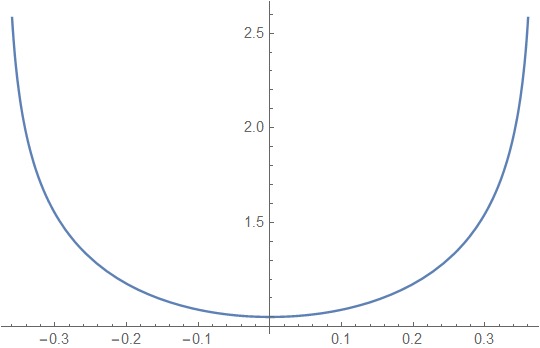}
\centering
\caption{Graph of solution $\rho(t)$ for example 2.}
\end{figure}

\end{example}

%

\bibliographystyle{amsplain}

\end{document}